\documentclass[11pt,a4paper]{article}
\usepackage[latin1]{inputenc}
\usepackage{a4wide}
\usepackage{amsmath}
\usepackage{amsthm}
\usepackage{amssymb}  
\usepackage{amsfonts}
\usepackage{cite}
\usepackage[pdftex]{graphicx}
\newtheorem{thm}{Theorem}
\newtheorem{lem}{Lemma}
\newtheorem{cor}{Corollary}

\newtheorem{prop}{Proposition}

\newtheorem{rem}{Remark}

\title{Complex supermanifolds of low odd dimension\\ and the example of the complex projective line}
\author{ Matthias Kalus\\
\small Fakult\"at f\"ur Mathematik\\ 
\small Ruhr-Universit\"at Bochum,\\ 
\small Universit\"atsstra\ss e 150\\ 
\small D-44801 Bochum, Germany\\
\small Matthias.Kalus@rub.de}
\date{}
\begin{document}
\maketitle
\begin{abstract} \noindent
Complex supermanifold structures being deformations of the exterior algebra of a holomorphic vector bundle, have been parametrized by orbits of a group on non-abelian cohomology (see \cite{Gr}). For the case of odd dimension $4$ and $5$ an identification of these cohomologies with a subset of abelian cohomologies being computable with less effort, is provided in this article. Furthermore for a rank $\leq 3$ sub vector bundle $F\to M$ of a holomorphic vector bundle $E=F\oplus F^\prime\to M$, a reduction of a (possibly non-split) supermanifold structure associated with $\Lambda E$ to a structure associated with $\Lambda F$ is defined. In the case of $rk(F^\prime)\leq 2$ with no global derivations increasing the $\mathbb Z$-degree by $2$, the complete cohomological information of a supermanifold structure associated with $E$ is given in terms of cohomologies compatible with the decomposition of $E$. Details on supermanifold structures of odd dimension 3 and 4 associated with sums of line bundles of sufficient 
negativity on 
$\mathbb P^1(\mathbb C)$ are deduced. 
\end{abstract}

\noindent
Complex non-split supermanifolds arise as deformations of a split complex supermanifold $(M,\mathcal O_{\Lambda E})$ constructed from a complex vector bundle $E\to M$. They are parametrized by orbits of the  group of bundle automorphism $H^0(M,Aut(E))$  on a certain in general non-abelian cohomology $H^1(M,G_E)$ (see \cite{Gr}). The cochains of this cohomology can be expressed as the exponential of elements in a certain abelian cochain complex $C^1(M,Der^{(2)}(\Lambda E))$ (see \cite{Ro}). However the degree of the involved finite exponential series is $k$ for $E$ of rank $2k$ or $2k+1$, increasing the complexity of computations for every second step in odd dimension. In particular $H^1(M,G_E)$ is zero up to odd dimension $1$, abelian up to odd dimension $3$ and in general non-abelian beyond this limit. 

\bigskip\noindent 
A method for relating supermanifolds of higher odd dimension to abelian cohomologies can hence considerably simplify computations. We restrict our view to the case of no global second degree derivations, i.e. $H^0(M,Der_2(\Lambda E))=0$. The main object of interest in this article is relating the lowest dimensional non-abelian cases of odd dimension $4$ and $5$ to the abelian case. For this the non-abelian cohomology classifying supermanifold structures of odd dimension $4$ or $5$ is embedded as a subset into the abelian cohomology $H^1(M,Der^{(2)}(\Lambda E))$ (second section).  This inclusion  depends on a fixed map $D$ relating cochains with values in a subsheaf $Der_2(\Lambda E)\subset Der^{(2)}(\Lambda E)$ that are appropriate for building supermanifolds, to cochains with values in a transversal complement $Der_4(\Lambda E)\subset Der^{(2)}(\Lambda E)$. However the image of the inclusion and the $H^0(M,Aut(E))$-action on it do not depend on the choice of $D$. A different way of relating elements in $H^1(
M,G_E)$ 
to cochains in an abelian complex was established in \cite{On-C} using a smooth Hermitian metric on $E$ and Hodge 
theory. 

\bigskip\noindent 
Furthermore it is proved that a reduction of the odd dimension is in general well defined for any subbundle $F \subset E$ of rank  $\leq 3$ with complement, i.e. $E=F\oplus F^\prime$. For the case $rk(F^\prime)\leq 2$ the cohomological data for a classification of supermanifold structures is given in terms of a sum of abelian cohomologies defined compatibly with the decomposition of $E$ (third section). This decomposition of cohomologies is of good use for the analysis of the $H^0(M,Aut(E))$-orbit structure on $H^1(M,G_E)$ since it is preserved by $H^0(M,Aut(F)\times Aut(F^\prime))$.

\bigskip\noindent 
Details on the orbit structure on $H^1(M,G_E)$ with respect to a maximal compact subgroup of $H^0(M,Aut(E))$ are deduced for rank 3 and 4 vector bundles being sums of line bundles of sufficient negativity on $\mathbb P^1(\mathbb C)$ (fourth section). Parameter spaces for special examples of supermanifold structures of odd dimension $3$ and $4$ on $\mathbb P^1(\mathbb C)$ were discussed and classified before in \cite{Ba1},\cite{Ba2},\cite{BO2}, \cite{BO}, \cite{Vi2}, and \cite{Vi1}. The case of odd dimension $2$ can be found in \cite{Vi1}. 

\section{Non-Split supermanifold structures}
The first section contains an introduction to the topic of complex non-split supermanifolds fixing the notation. Details can be found e.g. in \cite{Gr} and \cite{Ro}. 

\bigskip\noindent 
Let $\mathcal M=(M,\mathcal O_\mathcal M)$ be a complex supermanifold with underlying complex manifold $M$, sheaf of superfunctions $\mathcal O_\mathcal M$ and projection onto numerical holomorphic functions $pr:\mathcal O_\mathcal M \to \mathcal O_M$. Setting $\mathcal O_\mathcal M^{nil}:=Ker(pr)$ the sheaf  $\mathcal O_\mathcal M^{nil} /(\mathcal O_\mathcal M^{nil})^2$ defines a holomorphic vector bundle $E$ on $M$. Denote its automorphisms by $Aut(E)$, its sheaf of sections by $\mathcal O_E$, its full exterior power by $\Lambda E$ and the sheaf of automorphisms of algebras on $\mathcal O_{\Lambda E}$ preserving the $\mathbb Z/2\mathbb Z$-grading (but not necessarily the $\mathcal O_M$-module structure) by $Aut({\Lambda E})$. The rank of $E$ is the odd dimension of $\mathcal M$. Following \cite{Gr} denote by  $G_E\subset Aut({\Lambda E})$ the subsheaf of groups given by elements $\varphi \in Aut({\Lambda E})$ satisfying 
\begin{align*}
 (\varphi-Id)({\mathcal O_{\Lambda^j E}})\subset \textstyle{\bigoplus}_{k\geq 1} \mathcal O_{\Lambda^{j+2k} E} \ \qquad \forall \ j\geq 0 \ .
\end{align*}
It is proved in \cite{Gr} that the isomorphy classes of complex supermanifolds associated with a given vector bundle $E \to M$ are in $1:1$ correspondence to the $H^0(M,Aut(E))$-orbits by conjugation on the \v{C}ech cohomology $H^1(M,G_E)$. Note that this cohomology is meant with respect to composition of maps with identity as neutral element. So $H^1(M,G_E)$ is nothing more but a pointed set. The orbit of the identity in $H^1(M,G_E)$ corresponds to the unique split supermanifold structure associated with $E \to M$ given by $\mathcal O_\mathcal M=\mathcal O_{\Lambda E}$.

\bigskip\noindent
Following \cite{Ro} let $Der^{(2)}(\Lambda E)$  denote the sheaf of even derivations on the sheaf of $\mathbb Z/2\mathbb Z$-graded algebras $\mathcal O_{\Lambda E}$ satisfying 
$$w(\mathcal O_{\Lambda^j E})\subset \textstyle{\bigoplus}_{k\geq 1} \mathcal O_{\Lambda^{j+2k} E} \ \qquad \forall \ j\geq 0 \ . $$  
It is shown in \cite{Ro} that the exponential map maps $Der^{(2)}(\Lambda E)$ isomorphically onto $G_E$. The sheaf $Der^{(2)}(\Lambda E)$ itself decomposes  
\begin{align}\label{Der}
 Der^{(2)}(\Lambda E) = \textstyle{\bigoplus}_{k=1}^\infty Der_{2k}(\Lambda E) \  ,
\end{align}
where $Der_{2k}(\Lambda E)$ is the sheaf of even derivations satisfying $w(\mathcal O_{\Lambda^j E})\subset  \mathcal O_{\Lambda E^{j+2k}}$ for all $j\geq 0$.  The appropriate cohomology on  $Der^{(2)}(\Lambda E)$  is the usual abelian \v{C}ech cohomology with respect to the $\mathcal O_M$-module structure. 
\begin{rem}\label{rem:01}
(1) In odd dimension $0$ and $1$ the sheaf $Der^{(2)}(\Lambda E)$ is trivial and hence there are only split supermanifold structures.\\ 
(2) In odd dimension $2$ and $3$ it is $Der^{(2)}(\Lambda E)=Der_2(\Lambda E)$. The exponential mapping is just adding the identity and the composition in $Der_2(\Lambda E)$ is zero. Hence the supermanifold structures on $M$ associated with $E$ correspond to the orbits of $H^0(M,Aut(E))$ by conjugation on $H^1(M,Der_2(\Lambda E))$.  \\
(3) In odd dimension $\geq 4$ the cohomologies $H^1(M,G_E)$ and $H^1(M,Der^{(2)}(\Lambda E))$ are in general not isomorphic any more.
\end{rem}
\noindent
In the following $\mathbf d$ denotes the coboundary operator of the non-abelian cochain complex of $G_E$, while $d$ denotes the respective operator for the abelian complex of $Der^{(2)}(\Lambda E)$.

\section{Non-abelian cohomology in odd dimension 4 and 5}\label{sec:2}
The non-abelian cohomology $H^1(M,G_E)$ for $E$ of rank $4$ or $5$, is identified with a subset of the abelian cohomology  $H^1(M,Der^{(2)}(\Lambda E))$.

\bigskip \noindent 
In this section fix $rk(E) \in \{4,5\}$. For a cocycle $\exp(u_2+u_4) \in Z^1(M,G_E)$ where $u_{2s} \in C^1(M,Der_{2s}(\Lambda E))$, it is by direct calculation necessary that  $u_2 \in Z^1(M,Der_2(\Lambda E))$. Furthermore  define
\begin{align*}
 c_{u_2}=pr_{End_4(\Lambda E)}(\mathbf{d} \exp(u_2)) \in C^2(M,End_4(\Lambda E))\ ,
\end{align*}
where the notion of $End^{(2)}(\Lambda E)=\bigoplus_{k=1}^\infty End_{2k}(\Lambda E)$ in the sheaf of complex linear endomorphisms of $\mathcal O_{\Lambda E}$ is  defined analogously to (\ref{Der}). From the cocycle condition on $\exp(u_2+u_4)$ it follows that $c_{u_2}=-du_{4}$ and hence $c_{u_2}$ is a coboundary of derivations. 
Denote: 
\begin{align*}
 &\tilde  Z^1(M,Der_2(\Lambda E)):=\{u_2 \in Z^1(M,Der_2(\Lambda E)) \ | \ c_{u_2} \in B^2(M,Der_4(\Lambda E)) \} 
\end{align*}
For later application note that for cochains $u_2+u_4\in C^1(M,Der^{(2)}(\Lambda E))$ and $v_2+v_4 \in C^0(M,Der^{(2)}(\Lambda E))$ it follows by direct calculation that:\footnote{Here  $\exp(v_2+v_4).\exp(u_2+u_4)$ denotes $\big(\exp(v_{2,i}+v_{4,i})\exp(u_{2,ij}+u_{4,ij})\exp(-v_{2,j}-v_{4,j})\big)_{ij}$.}
\begin{align}\label{x3}
 \exp(v_2+v_4).\exp(u_2+u_4)=\exp(u_2+dv_2+u_4+dv_4+F(u_2,v_2))\\ \quad \mbox{with } F(v_2,u_2):=\frac{1}{2}([v_{2,i}+v_{2,j},u_{2,ij}]-[v_{2,i},v_{2,j}])_{ij} \nonumber
\end{align}
A map $D:H^0(M,Aut(E))\times C^0(M,Der_2(\Lambda E))\times \tilde Z^1(M,Der_2(\Lambda E))\to C^1(M,Der_4(\Lambda E))$ is called compatible if it satisfies:   
\begin{align}\label{x1}
 &d(D(\varphi,v_2,u_2))=c_{u_2} \\ \label{x2}
 \mbox{and }\qquad &\varphi.D(\varphi,v_2,u_2)=D(Id,0,(\varphi.u_2)+d(\varphi.v_2))+F(\varphi.v_2,\varphi.u_2) 
\end{align}
for all  $\varphi \in H^0(M,Aut(E))$, $v_2 \in C^0(M,Der_2(\Lambda E))$ and $u_2\in\tilde Z^1(M,Der_2(\Lambda E))$. It is called strongly compatible if additionally $D$ satisfies
\begin{align*}
 \varphi.D(\varphi,v_2,u_2)=D(Id,0,\varphi.u_2)
\end{align*}
for all allowed $(\varphi,v_2,u_2)$.

\begin{lem}\label{lemlem}
 A compatible map $D$ always exists. If $H^0(M, Der_2(\Lambda E))=0$, $D$ can be chosen to be strongly compatible.
\end{lem}
\begin{proof}
Let $D(Id,0,\ \cdot\ ):\tilde Z^1(M,Der_2(\Lambda E))\to C^1(M,Der_4(\Lambda E))$ be any map satisfying (\ref{x1}) for the third argument. This exists due to the definition of $\tilde Z^1(M,Der_2(\Lambda E))$. Continue it via (\ref{x2}) to $H^0(M,Aut(E))\times C^0(M,Der_2(\Lambda E))\times \tilde Z^1(M,Der_2(\Lambda E))$. From (\ref{x3}) setting $u_4=v_4=0$ it follows by direct calculation that
\begin{align}\label{eq-25}
 c_{u_2}=c_{u_2+dv_2}+dF(v_2,u_2)
\end{align}
for all $v_2 \in  C^0(M,Der_2(\Lambda E))$ and $u_2\in\tilde Z^1(M,Der_2(\Lambda E))$. Using this and (\ref{x2}), (\ref{x1}) holds for $\varphi=Id$. Since $u_2\mapsto c_{u_2}$, $d$ and $F$ are $H^0(M,Aut(E))$-equivariant, (\ref{x1}) holds for the continued map $D$. 

\smallskip\noindent
In the case $H^0(M, Der_2(\Lambda E))=0$ we choose $D(Id,0,\cdot)$ as above but require
\begin{align}\label{rep}
 D(Id,0,(\varphi.u_2)+d(\varphi.v_2))=D(Id,0,\varphi.u_2)-F(\varphi.v_2,\varphi.u_2)
\end{align}
for all allowed $(\varphi,v_2,u_2)$. This can be satisfied due to (\ref{eq-25}) and since $d(\varphi.v_2)=0$ implies $\varphi.v_2=0$ and so $F(\varphi,v_2,\varphi.u_2)=0$. Now proceed as above to obtain $D(\varphi,v_2,u_2)$. We find with (\ref{x2}) that the additional requirement yields strong compatibility of $D$.
\end{proof}

\noindent
In the following let $H^0(M, Der_2(\Lambda E))=0$ and $D$ will be chosen to be strongly compatible.

\begin{lem}\label{lem4}
 A strongly compatible $D$ induces a well-defined map 
 \begin{align*}
  \sigma_D:H^1(M,G_E)\to H^1(M,Der_2(\Lambda E)) \oplus H^1 (M,Der_4(\Lambda E))\
 \end{align*}
 given by $\sigma_D([\exp(u_2+u_4)])=([u_2],[D(Id,0,u_2)+u_4])$.   If $D$ additionally satisfies 
\begin{align}\label{eq:e}
 \varphi.D(Id,0,u_2)-D(Id,0,\varphi.u_2)\in B^1(M,Der_4(\Lambda E))
\end{align}
for all allowed $(\varphi,u_2)$ then the map $\sigma_D$ is $H^0(M,Aut(E))$-equivariant. 
\end{lem}
\begin{proof}
Note that if $\exp(u_2+u_4)\in Z^1(M,G_E)$ then $D(\varphi,v_2,u_2)+u_4 \in Z^1 (M,Der_4(\Lambda E))$. Further it is with (\ref{x3}) and $H^0(M,Aut(E))$-equivariance of $\exp$, $d$ and $F$:
\begin{align*}
 &\sigma_D([\varphi.(\exp(v_2+v_4).\exp(u_2+u_4))])\\& \qquad =\big([(\varphi.u_2)+d(\varphi.v_2)],[D(Id,0,(\varphi.u_2)+d(\varphi.v_2))+(\varphi.u_4)+d(\varphi.v_4)+F(\varphi.v_2,\varphi.u_2)]\big)
\end{align*}
The representing elements differ from those of $\varphi.\sigma_D(\exp(u_2+u_4))$ via equation (\ref{x2})  only by $(d(\varphi. v_2),D(Id,0,\varphi.u_2)-\varphi.D(Id,0,u_2)+d(\varphi.v_4))$. So $\sigma_D$ is well-defined for $\varphi=Id$ and under the additional requirement on $D$ also $H^0(M,Aut(E))$-equivariant.
\end{proof}

\noindent
Note that $\sigma_D$ is a bijection onto $\left(\tilde Z^1(M,Der_2(\Lambda E))/B^1(M,Der_2(\Lambda E))\right) \oplus H^1 (M,Der_4(\Lambda E))$, the first summand being well-defined by (\ref{eq-25}). So from Lemmas \ref{lemlem} and \ref{lem4} we can conclude: 
\begin{prop}\label{prop:02}
In the case $H^0(M, Der_2(\Lambda E))=0$, there always exists a bijection:  $$ \sigma_D:H^1(M,G_E)\to \left(\tilde Z^1(M,Der_2(\Lambda E))/B^1(M,Der_2(\Lambda E))\right) \oplus H^1 (M,Der_4(\Lambda E))$$ 
If a strongly compatible $D$ satisfying (\ref{eq:e}) exists then $\sigma_D$ is $H^0(M,Aut(E))$-equivariant.
\end{prop}

\section{Cohomology for decomposable vector bundles}

Assume in this section that $F \subset E$ is a complex sub vector bundle of rank $\leq3$ with $E=F\oplus F^\prime$ as vector bundles. In the first part of this section $E$ may have any rank $\geq rk{(F)}$ and the  projection morphism  $pr_F:E\to F$ is extended to $\Lambda pr_F:\mathcal O_{\Lambda E} \to \mathcal O_{\Lambda F}$. The goal is a restriction of a supermanifold structure on $E$ to a supermanifold structure on $F$, and secondly expressing the cohomological data in the case $rk(F^\prime)\leq 2$ in terms of abelian cohomologies compatible with the decomposition.

\bigskip\noindent
Let $[\alpha] \in H^1(M,G_E)$ be a cohomology class represented by $\alpha\in Z^1(M,G_E)$. Denoting by $End(\Lambda F)$ the sheaf of endomorphisms of the sheaf of complex vector spaces $\mathcal O_{\Lambda F}$ define  $$\alpha_F:=\Lambda pr_F \circ \alpha|_{\Lambda F} \in C^1(M,End({\Lambda F})) \ .$$ This cochain induces a supermanifold structure associated with  $F$:
\begin{lem}\label{lem:123}
 The cochain $\alpha_F$ lies in $Z^1(M,G_F)$  and the map 
 $$H^1(M,G_E)\longrightarrow H^1(M,G_F), \qquad [\alpha] \longmapsto [\alpha_F]  $$
 is well-defined.
\end{lem}
\begin{proof}
 Writing $\alpha_{ij}$ as the exponential of $\sum_{k=1}^\infty u_{2k,ij}$ with $u_{2k,ij} \in Der(M,Der_{2k}(\Lambda E))$ yields $(\alpha_F)_{ij}$ as the exponential of $\Lambda pr_F \circ u_{2,ij}|_{\Lambda F}$ since the $Der_4(\Lambda F)$-term vanishes. Hence $(\mathbf d(\alpha_F))_{ijk}= \Lambda pr_F\circ (Id+u_{2,ij}+u_{2,jk}-u_{2,ik} )|_{\Lambda F}=\Lambda pr_F\circ (\mathbf d(\alpha))_{ijk}|_{\Lambda F}=0$. In a similar way it is obtained that the map $\alpha \to \alpha_F$ maps coboundaries to coboundaries.
\end{proof}
\begin{rem}
(1) Note that $\alpha_F$ in general does not define the structure of a subsupermanifold.\\
 (2) It was used that the composition of endomorphisms, that increase the degree by $2$, is zero up to odd dimension $3$. The Lemma does in general not hold for higher rank subbundles.\\
 (3) In the special case\footnote{This case was pointed out to be of special interest in \cite{Ro}.} that the cocycle $\alpha$ can be chosen such that $\log(\alpha)\in Der^{(2\ell)}(\Lambda E)$, the Lemma follows for subbundles up to rank $4\ell-1$. 
\end{rem}

\noindent
Approaching odd dimension $4$ and $5$, from now on assume the case  $E=F \oplus F^\prime$ with $rk(F)\leq 3$ and  $rk(F^\prime)\leq 2$. This yields a decomposition 
$$\Lambda E=X\oplus Y \oplus Z, \quad \mbox{with } X=\Lambda F, \ Y=\Lambda F \otimes F^\prime, \ Z=\Lambda F \otimes \Lambda^2 F^\prime $$
Denote for $S,T \in  \{X,Y,Z\}$ by $Hom(T,S)$ the sheaf of homomorphisms of sheaves of complex vector spaces from $\mathcal O_T$ to $\mathcal O_S$ and set $End(T):=Hom(T,T)$ and $\tilde G_{T}:=\exp(End_2(T))$. For a cochain $\alpha \in C^1(M,G_E)$ regard the cochains
\begin{align}\label{ele} \begin{array}{l}
  \alpha_T:=(pr_{T}\circ\alpha|_{T}) \in C^1(M,\tilde G_{T}) \ \mbox{ for } \ T=X,Y,Z , \\
  u_{2,XY}+u_{4,XY}:=(pr_{Y} \circ \alpha|_{X}) \in C^1(M,Der_2(X,Y))\oplus C^1(M,Der_4(X,Y)) \ ,\\
  u_{2,XZ}+u_{4,XZ}:=(pr_{Z} \circ \alpha|_{X}) \in  C^1(M,Der_2(X,Z))\oplus C^1(M,Der_4(X,Z)) \ ,\\ 
  u_{2,YZ}+u_{4,YZ}:=(pr_{Z} \circ \alpha|_{Y}) \in  C^1(M,Hom_2(Y,Z))\oplus C^1(M,Hom_4(Y,Z)) \ ,\\
  u_{2,YX}:=(pr_{X} \circ \alpha|_{Y}) \in  C^1(M,Hom_2(Y,X)) \ ,\\
  u_{2,ZY}:=(pr_{Y} \circ \alpha|_{Z}) \in C^1(M,Hom_2(Z,Y)) \ . \end{array}
\end{align}
Note that  the  term $(pr_{X} \circ \alpha|_{Z})$ missing in the list, vanishes for reasons of degree. All eleven mentioned cochain complexes, those of the first line with respect to composition, the remaining with respect to the sum of maps,  are abelian. Continuing  all eleven cochains by zero on the complement of their domain of definition respectively, their sum equals $\alpha$. It follows from Proposition \ref{prop:02} and arguments similar to those in the proof of Lemma \ref{lem:123}:

\begin{prop}\label{prop:09}
Let $M$ be a complex manifold endowed with the sum of a rank $\leq 3$ vector bundle $F$ and a rank $\leq 2$ vector bundle $F^\prime$ denoted $E=F \oplus F^\prime$ with $H^0(M, Der_2(\Lambda E))=0$. Fix a strongly compatible map $D$  as in section \ref{sec:2} and decompose $D=D_{XY}+D_{XZ}+D_{YZ}$ with $D_{PQ}(\varphi,v_2,u_2):=pr_Q\circ D(\varphi,v_2,u_2)|_P$. The map of cochains  
\begin{align*}
  \alpha=\exp(u) \mapsto \Big(\alpha_T,\ u_{2,RS},\ D_{PQ}(Id,0,u_2)+u_{4,PQ} \ \Big| \ \begin{smallmatrix}{\scriptstyle R,S,T \in \{X,Y,Z\},\ R\neq S, \ (R,S)\neq(Z,X)}\\ (P,Q)\in \{(X,Y),(X,Z),(Y,Z)\}\end{smallmatrix} \Big)
\end{align*}
induces a map of cohomologies from $H^1(M,G_E)$ to the direct sum  ${\scriptstyle\bigoplus} H$ of the eleven abelian cohomologies of the cochain complexes in (\ref{ele}).  The induced map yields a bijection between $H^1(M,G_E)$ and the subset of  elements in ${\scriptstyle\bigoplus} H$  that can be represented by  cocycles  of the type $( \hat\alpha_T,\ \hat u_{2,RS},\ \hat u_{4,PQ})$ satisfying $\hat u_2:=\sum_T log(\hat \alpha_T)+\sum_{R,S} \hat u_{2,RS} \in Z^1(M,Der_2(\Lambda E))$ and $\hat u_4:=\sum_{R,S} \hat u_{4,RS}\in Z^1(M,Der_4(\Lambda E))$ as well as $c_{\hat u_2} \in B^2 (M,Der_4(\Lambda E))$.
\end{prop}
 
\begin{rem}\label{rem1} If $D$ also satisfies (\ref{eq:e}) then the inclusion $H^1(M,G_E) \hookrightarrow {\scriptstyle\bigoplus} H$ is by Proposition  \ref{prop:02} equivariant under the action of  $H^0(M,Aut(F)\times Aut(F^\prime))\subset H^0(M,Aut(E))$ acting diagonally on  ${\scriptstyle\bigoplus} H$. 
\end{rem}

\noindent 
For the case of $rk(F^\prime)=1$ denoting the line bundle $F^\prime$ by $L$, the result of Proposition \ref{prop:09} can be simplified. Most of the cochains in (\ref{ele}) vanish. The remaining are:
$$ \alpha_F:=\alpha_X, \quad \alpha_L:=\alpha_Y, \quad u_F:=u_{2,YX} \ \mbox{ and } \ u_L=u_{2,L}+u_{4,L}:=u_{2,XY}+u_{4,XY}$$
Note that for $f \in \mathcal O_M$, $s \in \mathcal O_L$ it is $u_F(fs)=pr_X((\alpha_F+u_{L})(f)s+f(\alpha_L+u_F)(s))=fu_F(s)$. Hence we can replace $H^1(M,Hom_2(Y,X))$ by $H^1(M,Hom_{\mathcal O_M}(L,\Lambda^3 F))$. Further we replace the cohomology $H^1(M,Der_4(X,Y))$ by $ H^1 (M,Der(\mathcal O_M)\otimes \mathcal O_{\Lambda^3 F \otimes L})$. 

\begin{cor}\label{cor:04}
For a complex manifold $M$ and the sum of a rank $\leq 3$ vector bundle $F$ and a line bundle $L$ denoted $E=F \oplus L$ such that $H^0(M, Der_2(\Lambda E))=0 $, fix a strongly compatible $D$. Then the elements $[\alpha]$ in the cohomology $H^1(M,G_E)$ correspond bijectively to the well defined  classes
\begin{align*}
 &([\alpha_F],[\alpha_L],[u_F],[u_{2,L}],[D(Id,0,u_2)+u_{4,L}]) \in  H^1(M,G_F) \oplus H^1(M,\tilde G_{\Lambda F\otimes L}) \\ &\quad  \oplus H^1(M,Hom_{\mathcal O_M}(L,\Lambda^3 F)) \oplus H^1(M,Der_2(\Lambda F,\Lambda F\otimes L) ) \oplus H^1 (M,Der(\mathcal O_M)\otimes \mathcal O_{\Lambda^3 F \otimes L})
 \end{align*} 
satisfying the two properties $u_2:=\alpha_F+\alpha_L+u_F+u_{2,L}-Id_{\Lambda E}\in Z^1(M,Der_2(\Lambda E))$ and  $c_{u_2} \in B^2(M,Der(\mathcal O_M)\otimes \mathcal O_{\Lambda^3 F \otimes L})$.
\end{cor}

\begin{rem}\label{remx}
The condition $u_2\in Z^1(M,Der_2(\Lambda E))$ contains that $\alpha_L$ is well-defined by $\alpha_F$, $u_F$ and $u_{2,L}$ and a term in $Z^1(M,Hom(L,\Lambda^2 F\otimes L))$. In particular  for $f \in \mathcal O_M$, $s \in \mathcal O_L$ it is $\alpha_L(fs)=pr_Y((\alpha_F+u_{L})(f)s+f(\alpha_L+u_F)(s))=\alpha_F(f)s+f\alpha_L(s)$.  So fixing $\alpha_F$, two possible choices of $\alpha_L$ differ by an element in $Z^1(M,Hom_{\mathcal O_M }(L,\Lambda^2 F\otimes L))$. This allows to regard the freedom in $H^1(M,\tilde G_{\Lambda F\otimes L})$ as a freedom in $H^1(M,Hom_{\mathcal O_M }(L,\Lambda^2 F\otimes L))$.
\end{rem}

\section{Examples on $\mathbb P^1(\mathbb C)$}
\noindent We discuss the orbit structure on $H^1(M,G_E)$ with respect to a maximal compact subgroup of $H^0(M,Aut(E))$ for the underlying manifold $M=\mathbb P_1(\mathbb C)$ for a large class of vector bundles. The example is studied here using Proposition \ref{prop:02} and Corollary \ref{cor:04}. We start with some general technical details and additional notation.

\bigskip\noindent Let $\mathcal O(k)$ for $k\in \mathbb Z$ denote the line bundle on $\mathbb P^1(\mathbb C)$ with divisor $k\cdot [0:1]$.\footnote{Note that the sign convention here is opposite to the convention used e.g. in \cite{Vi1}. } Fixing the coordinate chart $\mathbb P^1(\mathbb C)\backslash \{[1:0]\} \to \mathbb C$, $[z_0:z_1] \mapsto z:=\frac{z_0}{z_1}$ and denoting by $\mathbb C[z]_{\leq l}$ the polynomials of degree $\leq l$ and $\{0\}$ if $l<0$, we can identify the complex vector spaces 
$H^0(M,\mathcal O(k))\cong \mathbb C[z]_{\leq k}$ 
  and $H^1(M,\mathcal O(k))\cong \frac{1}{z}\mathbb C[\frac{1}{z}]_{\leq -k-2}$. 
Further note that $Hom_{\mathcal O_M}(\mathcal O(i),\mathcal O(j))\cong \mathcal O(j-i)$. Any complex vector bundle on $\mathbb P^1(\mathbb C)$ can be decomposed into a direct sum of line bundles (see \cite{Gro}) which are each isomorphic to one of the $\mathcal O(k)$. For a given vector bundle $E\to M$ fix such a decomposition $\mathcal O_E=\bigoplus_{i=1}^m \mathcal O(l_i)$ with $l_i\in \mathbb Z$. We fix $l_1\leq\cdots\leq l_m$. 
For later considerations we fix the standard local frame $\xi_i$ for the standard bundle chart of $\mathcal O(l_i)$ on $\mathbb P^1(\mathbb C)\backslash \{[1:0]\}$ and denote by $(\frac{\partial}{\partial \xi_i})_i$ the dual frame to the local frame $(\xi_i)_i$ of $E$. Elements in $Der(\Lambda E)$ hence locally appear as linear combinations of terms $\xi^If\frac{\partial}{\partial z},\xi^If\frac{\partial}{\partial \xi_i}$ with holomorphic numerical $f$ and $I \in \{0,1\}^m$. 

\bigskip \noindent
Set $m_i=\#\{l_j\ | \ l_j=i \}$ and note that $H^0(M,Hom(\mathcal O(i),\mathcal O(j)))\cong H^0(M,\mathcal O(j-i))$ realized in the above coordinates by multiplication of polynomials $\mathbb C[z]_{\leq j-i} \times \mathbb C[z]_{\leq i} \to \mathbb C[z]_{\leq j}$.
The global sections in $Aut(E)$ decompose into a semidirect product of groups $H^0(M,Aut(E))\cong A(E) \ltimes N(E)$ with:
\begin{align*}
A(E)&:=\textstyle{{\rm X}_{i=-\infty}^\infty} GL(m_i,\mathbb C)\cong\textstyle{{\rm X}_{i=-\infty}^\infty} H^0(M,Aut(\mathcal O(i)^{m_i})) \\
N(E)&:=Id_{\mathcal O_{\Lambda E}}+\textstyle{\bigoplus_{i<j}}\big(Hom(\mathbb C^{m_i},\mathbb C^{m_j})\otimes \mathbb C[z]_{\leq j-i}\big)\\&\ \cong Id_{\mathcal O_{\Lambda E}}+\textstyle{\bigoplus_{i<j}} H^0(M,Hom(\mathcal O(i)^{m_i},\mathcal O(j)^{m_j}))
\end{align*}
Set $U(E):=\textstyle{{\rm X}_{i=-\infty}^\infty} U(m_i)\subset A(E)$. 
Denote by $\rho$ and $\rho^\ast$ the standard, resp.~dual (here inverse transposed) action of the group $H^0(M,Aut(E))$, resp. of subgroups, on the vector space $\textstyle{\bigoplus_{i=-\infty}^\infty} \mathbb C^{m_i}$.

\bigskip\noindent
Note further  that  for fixed $2k$ a derivation in $Der_{2k}(\Lambda E)$ is given by its values on the sections in the subbundle $\Lambda^0E \oplus \Lambda^1 E\subset \Lambda E$. The continuation of a homomorphism on $\Lambda^1 E\subset \Lambda E$ by Leibniz rule and trivially on $\mathcal O_M$, respectively the restriction of a derivation to $\Lambda^0E$ yield an exact sequence (see \cite{BO})
$$0 \to Hom_{\mathcal O_M}(\Lambda^1 E,\Lambda^{2k+1} E) \to Der_{2k}(\Lambda E) \to  Der(\Lambda^0 E, \Lambda^{2k} E) \to 0 $$
with $Der(\Lambda^0 E, \Lambda^{2k} E)=Der(\mathcal O_M) \otimes \mathcal O_{\Lambda^{2k} E}$.  The long exact sequence of cohomology yields:
\begin{align*}
 &\ldots \to  H^0(M,Der(\Lambda^0 E, \Lambda^{2k} E)) \to H^1(M,Hom_{\mathcal O_M}(\Lambda^1 E,\Lambda^{2k+1} E)) \to H^1(M,Der_{2k}(\Lambda E))\\ &\qquad \qquad \to  H^1(M,Der(\Lambda^0 E, \Lambda^{2k} E)) \to  H^2(M,Hom_{\mathcal O_M}(\Lambda^1 E,\Lambda^{2k+1} E)) \to \ldots
\end{align*}
In any case $H^2(M,Hom_{\mathcal O_M}(\Lambda^1 E,\Lambda^{2k+1} E))=0$ for reasons of the dimension. In this chapter we specialize on the case 
$H^0(M,Der(\Lambda^0 E, \Lambda^{2k} E))=H^0(M,Der(\mathcal O_M) \otimes \mathcal O_{\Lambda^{2k} E})=0$. Due to $Der(\mathcal O_M)=\mathcal O(2)$ this means $l_{m-1}+l_m<-2$. All appearing sheaves are coherent sheaves on the compact complex manifold $M$ so we obtain an exact sequence of finite dimensional complex vector spaces: 
\begin{align*}
 0\to H^1(M,Hom_{\mathcal O_M}(\Lambda^1 E,\Lambda^{2k+1} E)) \to H^1(M,Der_{2k}(\Lambda E)) \to  H^1(M,Der(\Lambda^0 E, \Lambda^{2k} E)) \to  0
\end{align*}
Note that in the case $H^1(M,Hom_{\mathcal O_M}(\Lambda^1 E,\Lambda^{2k+1} E))= H^1(M,Der(\Lambda^0 E, \Lambda^{2k} E))=\{0\}$, triviality of $H^1(M,Der_{2k}(\Lambda E))$ follows.  
Furthermore the above sequence is equivariant under the respective $H^0(M,Aut(E))$-actions. There is a metric  invariant under the maximal compact Lie subgroup $U(E)$ of  $A(E)$, on the finite dimensional vector space $H^1(M,Der_{2k}(\Lambda E))$. The orthogonal complement to  $ H^1(M,Hom_{\mathcal O_M}(\Lambda^1 E,\Lambda^{2k+1} E)) \subset H^1(M,Der_{2k}(\Lambda E))$ yields a $U(E)$-equivariant splitting:
\begin{align}\label{eq:05}
 H^1(M,Der_{2k}(\Lambda E)) \cong H^1(M,Der(\Lambda^0 E, \Lambda^{2k} E))  \oplus H^1(M,Hom_{\mathcal O_M}(\Lambda^1 E,\Lambda^{2k+1} E)) 
\end{align}
We will use this splitting in the following. 

\bigskip\noindent
In order to apply Proposition \ref{prop:02}, we need $H^0(M,Der_2(\Lambda E))=0$. Note that the above long exact sequence of cohomology also yields exactness of:
\begin{align}\label{dfg}
 H^0(M,Hom_{\mathcal O_M}(\Lambda^1E,\Lambda^{2k+1}E))\to H^0(M,Der_{2k}(\Lambda E))\to H^0(M,Der(\Lambda^0 E,\Lambda^{2k}E))
\end{align}
Now for $k\geq 1$:
\begin{align}\label{dfg2}
 Hom_{\mathcal O_M}(\Lambda^1E,\Lambda^{2k+1}E)\cong \mathcal O_{\Lambda^{2k+1}E}\otimes \mathcal O_E^\ast \quad \mbox{ and }\quad
Der(\Lambda^0 E,\Lambda^{2k}E)\cong \mathcal O_{\Lambda^{2k}E}\otimes \mathcal O(2)
\end{align}
So we assume in the following: 
$$l_{m-1}+l_m<-2 \qquad \mbox{and}\qquad l_{m-2}+l_{m-1}+l_m-l_1<0$$ Then the global sections of the sheaves in (\ref{dfg2}) vanish and (\ref{dfg}) yields $H^0(M,Der_2(\Lambda E))=0$.

\subsection*{Odd dimension 3}
\noindent
Assume now that $m=3$.  It is $H^1(M,G_E)\cong H^1(M,Der_2(\Lambda E))$. Using (\ref{eq:05}) and further that $Hom_{\mathcal O_M}(\Lambda^1E,\Lambda^3 E)$ consists of multiplication operators in $\mathcal O_{\Lambda^2 E}$, it is:
\begin{align*} 
 H^1(M,Der_2(\Lambda E)) & \cong \textstyle{\bigoplus}_{1 \leq i<j \leq 3}\left(H^1(M,\mathcal O(l_i+l_j+2))\oplus H^1(M,\mathcal O(l_i+l_j))\right) 
\end{align*}
Hence it is:
\begin{align}\label{stern} 
\begin{array}{ll}
& H^1(M,Der_2(\Lambda E))\cong {\textstyle\bigoplus_{1 \leq i<j \leq 3}}\left( \frac{1}{z}\mathbb C[\frac{1}{z}]_{\leq c_{ij}}\oplus \frac{1}{z}\mathbb C[\frac{1}{z}]_{\leq d_{ij}}\right) \quad \mbox{with}\vspace*{0.2cm}\\
&c_{ij}:=-l_i-l_j-4 \quad \mbox{and} \quad d_{ij}:=-l_i-l_j-2  \ 
\end{array}
\end{align}
It follows that the group $U(E)$  acts  on $H^1(M,Der_2(\Lambda E))$ by multiples of restrictions of the actions $\rho\wedge\rho$, resp. $\rho \wedge \rho\wedge \rho \otimes \rho^\ast$. 
Denote for the description of the $U(E)$-orbits, $$\mathbb V_{k}^n:=\{ U(k)(v) \ | \ v \in (\textstyle{\frac{1}{z}\mathbb C[\frac{1}{z}]}_{\leq n})^k \}$$ for $0\leq k, n$ and  $\mathbb V^{n}_k=\{0\}$ else. 
Identify $\frac{1}{z}\mathbb C[\frac{1}{z}]_{\leq c_{ij}}\oplus \frac{1}{z}\mathbb C[\frac{1}{z}]_{\leq d_{ij}}\cong\frac{1}{z}\mathbb C[\frac{1}{z}]_{\leq c_{ij}+d_{ij}+1}$ via the homomorphism $(p,q)\mapsto p+z^{-c_{ij}-1}q$. 

\begin{prop}\label{prop:dings}
 The $U(E)$-action on $H^1(M,Der_2(\Lambda E))$ has with respect to (\ref{stern}) the orbits:
 \begin{itemize}
 \item[] $\mathbb V_3^{c_{12}+d_{12}+1}$ in the case $l_1=l_2=l_3$,
 \item[] $ \mathbb V_1^{c_{12}+d_{12}+1} \times \mathbb V_2^{c_{13}+d_{13}+1}$ in the case $l_1=l_2<l_3$,
  \item[] $\mathbb V_2^{c_{12}+d_{12}+1}\times \mathbb V_1^{c_{23}+d_{23}+1}$ in the case $l_1<l_2=l_3$,
 \item[] $ \mathbb V_1^{c_{12}+d_{12}+1}\times \mathbb V_1^{c_{13}+d_{13}+1}\times \mathbb V_1^{c_{23}+d_{23}+1}$ in the case $l_1<l_2<l_3$.
\end{itemize}
\end{prop}
\begin{proof} 
\noindent\textit{Case  $l_1=l_2=l_3$:}
In this case it is $H^0(M,Aut(E))\cong A(E)= GL(3,\mathbb C)$. The action of $U(E)=U(3)$ on the vector space  $H^1(M,G_E)\cong (T({c_{12},d_{12}}))^3$ with $T({c_{12},d_{12}}):= \frac{1}{z}\mathbb C[\frac{1}{z}]_{\leq c_{12}}\oplus \frac{1}{z}\mathbb C[\frac{1}{z}]_{\leq d_{12}}$ is given in a suitable basis by $A \mapsto M_A \otimes Id_{T({c_{12},d_{12}})}$, where $M_A$ is the matrix of minors of $A$. Note further that it is  $\{M_A\ |\ A\in U(3)\}=U(3)$. \\
\noindent\textit{Case  $l_1=l_2<l_3$:}
In this case $U(E)=S^1 \times U(2)$. Its action on   $H^1(M,G_E)\cong  T({c_{12},d_{12}})\oplus (T ({c_{13},d_{13}}))^2 $
is  $$(A,d) \mapsto \left(\det(A) \cdot Id_{T({c_{12},d_{12}})}\right) \oplus \left(d\cdot det(A)\cdot (A^{-1})^T \otimes Id_{T({c_{13},d_{13}})}\right)$$ yielding orbits parametrized by $\mathbb V_1^{c_{12}+d_{12}+1} \times \mathbb V_2^{c_{13}+d_{13}+1} $. \textit{Case}  $l_1<l_2=l_3$ follows analogously.   \\
\noindent\textit{Case  $l_1<l_2<l_3$:}
In this case it is $U(E)= (S^1)^3$. Its action on the identified vector space  $H^1(M,G_E)\cong  T({c_{12},d_{12}})\oplus T({c_{13},d_{13}})\oplus T({c_{23},d_{23}})$
is given by $(\lambda_1,\lambda_2,\lambda_3) \mapsto \lambda_1\cdot \lambda_2 \cdot Id_{ T({c_{12},d_{12}})} \oplus \lambda_1\cdot \lambda_3 \cdot Id_{ T({c_{13},d_{13}})}\oplus \lambda_2\cdot \lambda_3 \cdot Id_{ T({c_{23},d_{23}})}$.
\end{proof}

 
\subsection*{Odd dimension 4}
In the case of $rk(E)=4$, the  $\mathcal O_M$-module $Hom_{\mathcal O_M}(\Lambda^1E,\Lambda^3 E)$ in (\ref{eq:05}) is generated  by  contractions in $\mathcal O_{E^\ast}$ followed by a multiplication operator in $\mathcal O_{\Lambda^3 E}$. Furthermore note that $H^1(M,Der_4(\Lambda E))\cong H^1(M,Der(\mathcal O_M)\otimes\mathcal O_{\Lambda^4 E})$.   It follows: 
 \begin{lem}\label{lem:test}
   The group $U(E)$  acts on $H^1(M,Der_2(\Lambda E))$ by multiples of restrictions of $\rho \wedge \rho$, resp. $\rho \wedge \rho\wedge \rho\otimes \rho^\ast$ according to the decomposition in (\ref{eq:05}). On $H^1(M,Der_4(\Lambda E))$ the group acts by the determinant.
 \end{lem}
 \noindent
Note that   $c_{u_2}=0$ for all $u_2 \in Z^1(M,Der_2(\Lambda E))$ by the covering chosen above. So we have $\tilde  Z^1(M,Der_2(\Lambda E))= Z^1(M,Der_2(\Lambda E))$ on $\mathbb P^1(\mathbb C)$. The canonical representatives $\chi$ of classes in $H^1(M,Der_2(\Lambda E))$ in the above coordinates are linear combinations of elements:
\begin{align*}
 &\textstyle{\frac{1}{z^{r+1}}}\xi_i\xi_j\frac{\partial}{\partial z} &&\mbox{ with } 0\leq r \leq -l_i-l_j-4  &&\mbox{ and }\\ &\textstyle{\frac{1}{z^{s+1}}}\xi_i\xi_j\xi_k\frac{\partial}{\partial \xi_t} &&\mbox{ with } 0\leq s \leq -l_i-l_j-l_k+l_t-2&&  
\end{align*}
Set $D(Id,0,\chi)=0$ for all of these canonical representatives. Then via (\ref{rep}), $D(Id,0,\cdot)$ can be defined on all of $Z^1(M,Der_2(\Lambda E))$. Now the map $D$  satisfies (\ref{eq:e}) for $\varphi\in U(E)$.  Continue $D$ to $U(E)\times C^0(M,Der_2(\Lambda E))\times Z^1(M,Der_2(\Lambda E))$ as in the proof of Lemma \ref{lemlem}. 

\bigskip \noindent
Proposition \ref{prop:02} yields an $U(E)$-equivariant bijection $\sigma_D: H^1(M,G_E)\to H^1(M,Der^{(2)}(\Lambda E))$ with the above $D$. We divide the general situation into three cases. 

\subsection*{Fourfold sum of a line bundle}
Assuming $\mathcal O_E=4\mathcal O(l)$ it is with respect to the decomposition in (\ref{eq:05}):
\begin{align*}
 &H^1(M,Der_2(\Lambda E))\cong H^1(M,\mathcal O(2l+2))^6 \oplus H^1(M,\mathcal O(2l))^{16}\\
 &H^1(M,Der_4(\Lambda E))=H^1(M,\mathcal O(4l+2)) \ 
\end{align*}
So  it follows from $H^0(M,Aut(E))\cong A(E)$, Proposition \ref{prop:02} and Lemma \ref{lem:test}:

\begin{thm}\label{thm:05}
The supermanifold structures associated with $E$ with $l<-1$ are parametrized by the $U(4)$-orbits of the diagonal action given by  $ \rho \wedge \rho$, $\rho \wedge \rho\wedge \rho\otimes \rho^\ast$ and the determinant on the three summands of the vector space:
$$ (\textstyle{\frac{1}{z}\mathbb C[\frac{1}{z}]}_{\leq-2l-4})^6 \ \oplus \ (\textstyle{\frac{1}{z}\mathbb C[\frac{1}{z}]}_{\leq-2l-2})^{16} \ \oplus \ \textstyle{\frac{1}{z}\mathbb C[\frac{1}{z}]}_{\leq-4l-4} $$
\end{thm}


\subsection*{Two couples}
Assume now that $\mathcal O_E=2\mathcal O(l)\oplus 2\mathcal O(l^\prime)$ with $l<  l^\prime$. Then $A(E)= GL(2,\mathbb C)\times GL(2,\mathbb C)$ and $N(E)=Id_E+Hom(\mathbb C^{2},\mathbb C^{2})\otimes \mathbb C[z]_{\leq l^\prime-l}$ as lower-left  block matrices. Denote the standard and determinant action  of the factors of $A(E)$ by $\rho_{i}$, resp.~$det_i$, $i=1,2$. It is with respect to the decomposition in (\ref{eq:05}) -- each line one term:
\begin{align*}
 H^1(M,Der_2(\Lambda E))\cong&\quad\ H^1(M,\mathcal O(2l+2)) \oplus (H^1(M,\mathcal O(l+l^\prime+2)))^4 \oplus H^1(M,\mathcal O(2l^\prime+2)) \\& \oplus H^1(M,\mathcal O(2l))^4 \oplus (H^1(M,\mathcal O(l+l^\prime)))^8 \oplus H^1(M,\mathcal O(2l^\prime))^4\\ H^1(M,Der_4(\Lambda E))\cong& \quad\ H^1(M,\mathcal O(2(l+l^\prime)+2)) \ 
\end{align*}
  Proposition \ref{prop:02} and Lemma \ref{lem:test} yield:

\begin{prop}\label{prop:06}
Non-split supermanifold structures on $\mathbb P^1(\mathbb C)$ associated with a vector bundle of the form $\mathcal O_E=2\mathcal O(l)\oplus 2\mathcal O(l^\prime)$ with $l<l^\prime$ only appear if $l\leq -1$. Identifying in the case $l^\prime< -1$ the cohomology $H^1(M,G_E)$ with 
\begin{align*}\label{asd}
\begin{array}{rllll} & \quad \textstyle{\frac{1}{z}\mathbb C[\frac{1}{z}]}_{\leq-2l-4} &\oplus\quad    (\textstyle{\frac{1}{z}\mathbb C[\frac{1}{z}]}_{\leq-l-l^\prime-4})^4  &\oplus\quad   \textstyle{\frac{1}{z}\mathbb C[\frac{1}{z}]}_{\leq -2l^\prime-4}  &\vspace*{0.1cm}  \\   \oplus\  & (\textstyle{\frac{1}{z}\mathbb C[\frac{1}{z}]}_{\leq-2l-2})^4 &\oplus \quad  (\textstyle{\frac{1}{z}\mathbb C[\frac{1}{z}]}_{\leq-l-l^\prime-2})^8  &\oplus \quad   (\textstyle{\frac{1}{z}\mathbb C[\frac{1}{z}]}_{\leq-2l^\prime-2})^4 &\oplus   \quad  \textstyle{\frac{1}{z}\mathbb C[\frac{1}{z}]}_{\leq-2l-2l^\prime-4} \end{array}
\end{align*}
the diagonal action of $U(E)=U(2)\times U(2)$ is given by $det_1$, $\rho_{1}\otimes \rho_{2}$, $det_2$, $det_1\cdot \rho_2 \otimes \rho_2^\ast$, $(det_1\cdot \rho_2 \otimes \rho_1^\ast)\oplus (det_2\cdot \rho_1 \otimes \rho_2^\ast)$, $det_2\cdot \rho_1 \otimes \rho_1^\ast$ and $det_1 \cdot det_2$. 
\end{prop}

\subsection*{A distinct line bundle}
Decompose a rank $4$ vector bundle $E=F\oplus L$ with $\mathcal O_F=\mathcal O(l_1)\oplus \mathcal O(l_2)\oplus \mathcal O(l_3)$  ordered to $l_1 \leq l_2\leq l_3$ and $\mathcal O_L=\mathcal O(l)$ with $l\neq l_i$ for all $i=1,2,3$. Note that we can assume without loss of generality that $l < l_1$ or $l>l_3$. Following Corollary \ref{cor:04} the relevant cohomologies involved in a classification of supermanifold structures are $H^1(M,G_F)$ given analogue to the case of 3 odd dimensions in  (\ref{stern}), $ H^1(M,Hom_{\mathcal O_M}(L,\Lambda^3 F)) \cong H^1(M,\mathcal O(l_1+l_2+l_3-l)) \cong \textstyle{\frac{1}{z}\mathbb C[\frac{1}{z}]}_{\leq c} $ with 
$$c:=  -l_1-l_2-l_3+l-2 $$
 and $H^1(M,Der_2(\Lambda F,\Lambda F \otimes L))$, $H^1(M,Der_4(\Lambda F,\Lambda F \otimes L))$ and $H^1(M,\tilde G_{\Lambda F \otimes L})$. By Remark \ref{remx} the only remaining relevant term in the last cohomology is generated by $id_{\mathcal O_L}$ followed by a multiplication in $\mathcal O_{\Lambda^2 F}$ yielding $\bigoplus_{1 \leq i,j \leq 3} H^1(M,\mathcal O(l_i+l_j))\cong \bigoplus_{1 \leq i,j \leq 3} \textstyle{\frac{1}{z}\mathbb C[\frac{1}{z}]}_{\leq d_{ij}^\prime}$ with $$  d_{ij}^\prime:=-l_i-l_j-2. $$
 It is with respect to the decomposition in (\ref{eq:05}) in the case $l_3+l_4<-2$, resp. $l_4+l<-2$:
\begin{align*}
&H^1(M,Der_2(\Lambda F,\Lambda F \otimes L))\cong \quad\ \textstyle{\bigoplus}_{1 \leq i \leq 3} \Big( H^1(M,\mathcal O(l_i+l+2))\\&\qquad\qquad\qquad\qquad\qquad\oplus H^1(M,\mathcal O(l_i+l))^2\oplus H^1(M,\mathcal O(l_1+l_2+l_3+l-2l_i)) \Big)
\end{align*}
So $H^1(M,Der_2(\Lambda F,\Lambda F \otimes L))\cong \bigoplus_{1 \leq i \leq 3}\left( \textstyle{\frac{1}{z}\mathbb C[\frac{1}{z}]}_{\leq c_{i}}\oplus (\textstyle{\frac{1}{z}\mathbb C[\frac{1}{z}]}_{\leq d_{i}})^2\oplus \textstyle{\frac{1}{z}\mathbb C[\frac{1}{z}]}_{\leq d_{i}^\prime}\right)$ with:
\begin{align*}
&c_{i}:=-l_i-l-4, \qquad   d_{i}:=-l_i-l-2, \qquad d_{i}^\prime:=-l_1-l_2-l_3-l+2l_i-2
\end{align*}
Finally $H^1(M,Der_4(\Lambda F,\Lambda F \otimes L))  \cong H^1(M,\mathcal O(l_1+l_2+l_3+l+2)) \cong \textstyle{\frac{1}{z}\mathbb C[\frac{1}{z}]}_{\leq d}$ with
$$d:= -l_1-l_2-l_3-l-4 $$
It is $H^0(M,Aut(E))\cong(A(F)\times A(L))\ltimes (N(F)\times N^\prime(E))$  with $A(F)\subset GL(3,\mathbb C)$, $A(L)=\mathbb C^\times$, $N(F)$ as in the case of odd dimension three and $N^\prime(E)=pr_{E\to L}+ \textstyle{\bigoplus_{i=1}^3}\mathbb C[z]_{\leq |l_i-l|}$ as  upper-right, resp. lower-left block matrices  if $l<l_1\leq l_2\leq l_3$, resp. if $l_1\leq l_2\leq l_3<l$.

\bigskip\noindent
Corollary \ref{cor:04}, Proposition \ref{prop:dings} and Lemma \ref{lem:test} yield:

\begin{prop}\label{prop:04}
Non-split supermanifold structures of odd dimension $4$ on $\mathbb P^1(\mathbb C)$ associated with a vector bundle $\mathcal O_E=\mathcal O(l_1)\oplus \mathcal O(l_2)\oplus \mathcal O(l_3)\oplus \mathcal O(l)$  with $l<l_1\leq l_2\leq l_3$, resp. $l_1\leq l_2\leq l_3<l$ only appear if $l+l_1\leq -2$, resp. $l_1+l_2\leq -2$. Identifying in the case $l_3+l_4<-2$, resp. $l_4+l<-2$ the cohomology $H^1(M,G_E)$ with 
\begin{align*}
\begin{array}{rlll}
 & {\textstyle\bigoplus_{1\leq i,j\leq 3} } \textstyle{\frac{1}{z}\mathbb C[\frac{1}{z}]}_{\leq c_{ij}} & \oplus\quad {\textstyle\bigoplus_{1\leq i\leq 3} } \textstyle{\frac{1}{z}\mathbb C[\frac{1}{z}]}_{\leq c_{i}}  & \oplus\quad  {\textstyle\bigoplus_{1\leq i\leq 3} } \Big( (\textstyle{\frac{1}{z}\mathbb C[\frac{1}{z}]}_{d_{i}})^2  \oplus  \textstyle{\frac{1}{z}\mathbb C[\frac{1}{z}]}_{\leq d_{i}^\prime}\Big)  \\ \oplus\ &   {\textstyle\bigoplus_{1\leq i,j\leq 3} }\textstyle{\frac{1}{z}\mathbb C[\frac{1}{z}]}_{\leq d_{ij}}  & \oplus\quad  \textstyle{\frac{1}{z}\mathbb C[\frac{1}{z}]}_{\leq c} & \oplus\quad   {\textstyle\bigoplus_{1\leq i,j\leq 3} } \textstyle{\frac{1}{z}\mathbb C[\frac{1}{z}]}_{\leq d^\prime_{ij}} \quad \oplus\quad  \textstyle{\frac{1}{z}\mathbb C[\frac{1}{z}]}_{\leq d} 
\end{array}
\end{align*}
the diagonal action of $U(F)\subset U(3)$ is given by the restrictions of $\rho\wedge \rho$ on the first and sixth, of $\rho$ on the second, of $\rho\wedge \rho\otimes\rho^\ast$ on the third, of $\rho\wedge \rho\wedge \rho\otimes\rho^\ast$ on the fourth and by the determinant action on the fifth and seventh summand. The diagonal  action of $A(L)=S^1$ is trivial on the first, fourth and sixth, dual on the fifth and standard on the three remaining summands. 
\end{prop}

\subsubsection*{Acknowledgments.} The author thanks an anonymous referee for pointing out several mistakes.

\bibliographystyle{9}

\end{document}